\documentclass{amsart}

\usepackage{amsmath,amsfonts,amsthm,tikz,amssymb,hyperref}

\usepackage[enableskew]{youngtab}

\theoremstyle{plain}
\newtheorem{theorem}{Theorem}[section]

\newtheorem{lemma}[theorem]{Lemma}
\newtheorem{corollary}[theorem]{Corollary}

\theoremstyle{definition}
\newtheorem{example}[theorem]{Example}
\newtheorem{remark}[theorem]{Remark}
\newtheorem{definition}[theorem]{Definition}

\usepackage{ytableau,tikz,varwidth}
\usetikzlibrary{calc}

\usepackage{amsmath,amsfonts,amsthm,tikz,amssymb,hyperref}

\usepackage[enableskew]{youngtab}

\numberwithin{equation}{section}
\numberwithin{table}{section}

\newcommand{\Sym}{\textit{Sym}}     %% Algebra of syymmetric functions
\newcommand{\D}[1]{D_{#1}}          %% Operator D
\newcommand{\ZZ}{\mathbb{Z}}             %% integer numbers

\renewcommand{\ss}{\sigma}              %% Generatingseries for the h functions
\begin{document}

%%%% TITLE

%[Commutation relations for operators]
\title[Combinatorics of the $\Gamma_{(t|X)}$ vertex operator]{A comment of the combinatorics of the  vertex operator $\Gamma_{(t|X)}$} 
%\date{\today}

%%%% AUTHORS AND ADDRESSES

\author[M.H. Rosas]{Mercedes Helena Rosas}
\address{Departamento de \'Algebra, Facultad de Matem\'aticas, Universidad de Sevilla, Avda.\ Reina Mercedes, Sevilla, Espa\~na}
\email{mrosas@us.es}

\thanks{M.\ Rosas has been partially supported by projects MTM2010--19336, MTM2013-40455-P, FQM--333, P12--FQM--2696 and FEDER.}

\maketitle

%%%%%%%%%%%%%%%%%%%%
%%%%%%%%%%%%%%%%%%%%
%%%%%%%%%%%%%%%%%%%%

\begin{abstract}
The Jacobi--Trudi identity associates a symmetric function to any integer sequence.
 Let $\Gamma_{(t|X)}$ be the  vertex operator defined by $\Gamma_{(t|X)} s_\alpha =\sum_{n \in \ZZ}  s_{(n,\alpha)} [X] t^n$. We provide a combinatorial proof  for the identity  $\Gamma_{(t|X)} s_\alpha = \sigma[tX] s_{\alpha}\big[x-1/t\big] $ due to Thibon et al., \cite{ScharfThibonWybourne, ThibonHopf}. We include an overview of all the combinatorial ideas behind this beautiful identity, including a combinatorial description for the expansion of $s_{(n,\alpha)} [X] $ in the Schur basis, for any integer value of $n$.
\end{abstract}

\section{introduction}

This short note concerns a  powerful operator on symmetric functions,  the vertex operator $\Gamma_{(t|X)}$. Let $\alpha$ be a partition, and let $s_\alpha$ be the corresponding  Schur function, then  $\Gamma_{(t|X)}$ is   the generating function  for  Schur functions indexed by those integer sequences obtained from $\alpha$ by  prepending an  integer part:
 \begin{align}
\Gamma_{(t|X)} s_\alpha =\sum_{n \in \ZZ} s_{(n,\alpha)} [X] t^n
\end{align}

%
%The vertex operator $\Gamma_{(t|X)}$ was introduced by Thibon in order to make sense of the classical notation...
%Then  Thibon et al. used $\Gamma_{(t|X)}$  to study stability properties for  Kronecker and plethysm coefficients,  \cite{ScharfThibonWybourne, ThibonHopf}.  In mathematical physics, they appear in the context of the boson--fermion correspondence,  \cite{ JimboMiwa, Kac},   as carefully explained in \cite{ThibonHopf}.
%
This vertex operator is a classical tool in the theory of symmetric functions.  It has been used by Thibon and his collaborators for the study of various phenomena of stability \cite{ ThibonCarre, ThibonLascoux:stability, ThibonHopf, ThibonScharf:innerPlethysm, ScharfThibonWybourne}. Previously, it had been used in mathematical physics, Jing  \cite{Jing:Spin}.  Likewise, $\Gamma_{(t|X)}$  is the generating series for Bernstein's creation operators introduced in \cite{Zelevinsky:LNM}. % See also \cite[I.\S 5 Ex. 29]{MacdonaldBook}. 
Recently,   $\Gamma_{(t|X)}$ has   played a central role in our  work on operators on symmetric functions \cite{BMOR}, and on the growth of the Kronecker coefficients \cite{BRR, BRR2}.
   
The aim of this note is to provide a combinatorial proof for the  identity  
\begin{align} \label{theidentity}
\Gamma_{(t|X)} s_\alpha = \sigma[tX] s_{\alpha}\big[x-1/t\big]
\end{align}
 by means of a sign reversing involution. While concise algebraic proof  can be found in  \cite{ScharfThibonWybourne, ThibonHopf}, readers with a more  combinatorial taste  may find this proof interesting.

We take the opportunity to present a compilation of  results (some of them very well-known, other less so) needed to interpret all functions appearing in Eq. 
\ref{theidentity} combinatorially. As usual, we follow the notation of Macdonald's book \cite{MacdonaldBook}, unless it is explicitely stated. A point of depart will be that  we draw our partitions the French way.
Lascoux's approach to symmetric functions permeates this note, \cite{LascouxBook}.
%\begin{align*}
%	\Gamma_{(t|X)} s_\alpha =\sum_{n \in \ZZ} s_{(n,\alpha)} [X] t^n\\
%	 \Gamma_{(t|X)}= \sigma[tX] \sigma[ -1/t \, \, x^\perp]
%\end{align*}

The famous  Jacobi--Trudi identity express Schur function in the complete homogeneous basis
% Explicitly, given a finite integer sequence $\alpha= (\alpha_1,\alpha_2,\ldots,\alpha_l)$,   we use the Jacobi--Trudi determinant to define $s_\alpha$:  
$
s_{(\alpha_1,\alpha_2,\ldots,\alpha_l)}=\det(h_{\alpha_{i}+j-i})_{i,j=1\ldots l}
$
where, as usual, the $h_k$'s are the complete homogeneous symmetric functions,   $h_0=1$ and $h_k=0$ when $k < 0$.

While it is a well--known theorem that the previous identity holds for any partition (see for example the beautiful combinatorial proof of Gessel and Viennot using lattive paths), 
the Jacobi-Trudi determinant makes perfectly good  sense for any integer sequence $\alpha$.  This allows us to make sense of the definition of the vertex operator $\Gamma_{(t|X)}$, that deals with  functions $s_\alpha$ whose first part is out-of-order or even nonpositive. 
%Moreover, when $\alpha$ is a partition (\emph{i.e.,} when $\alpha$ is a weakly decreasing sequence of  positive integers),
% the Jacobi--Trudi identity says that this determinant coincides with the Schur function $s_{\alpha}$ 
 It turns out that for any  integer sequence, the Jacobi--Trudi determinant $s_\alpha$ is equal to either zero, or to $\pm$ a Schur function $s_\lambda$, with $\lambda$ a partition. In this situation, we say that the partition $\lambda$  is the {\em rectification of $\alpha$. }

The process of rectifying a shape coming from an integer sequence is better described using an example. The identity  $s_{5,3,2,7}=s_{5,4,4,3}$  can be shown to hold writing the corresponding Jacobi-Trudi determinant, and performing row exchanges:
\begin{align*}
 \left| \begin{array}{cccc}
h_5 & h_6 & h_7 & h_8 \\
h_2 & h_3 & h_4 & h_5 \\
1 & h_1 & h_2 & h_3 \\
h_4 & h_5& h_6 & h_7 \\
\end{array} \right|
= 
(-)
\left| \begin{array}{cccc}
h_5 & h_6 & h_7 & h_8 \\
h_2 & h_3 & h_4 & h_5 \\
h_4 & h_5& h_6 & h_7 \\
1 & h_1 & h_2 & h_3 \\
\end{array} \right|
= 
(+)
\left| \begin{array}{cccc}
h_5 & h_6 & h_7 & h_8 \\
h_4 & h_5& h_6 & h_7 \\
h_2 & h_3 & h_4 & h_5 \\
1 & h_1 & h_2 & h_3 \\
\end{array} \right|
\end{align*}
Note that  the weight of the rectified partition coincides with weight of the original  integer sequence (i.e., the sum of the parts of the integer partition).
This process is sometimes known as Littlewood's modification rule.

%This operation can be performed directly on tableaux. 
%Recall that we use the French convention for drawing partitions/integer sequences. 
It is well-known that we can perform the row exchanges  on the Jacobi--Trudi determinants directly on the  tableaux indexing these determinant ``by letting gravity act". Draw the try diagrams corresponding to the Jacobi--Trudi identities appearing in our running example:
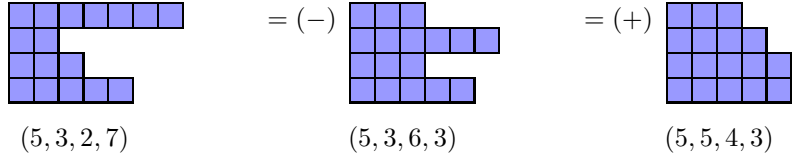
\begin{align*}
\ytableausetup
{boxsize=.9em}\ytableausetup
{aligntableaux=top}
\begin{ytableau}
*(blue!40)  & *(blue!40)   &*(blue!40)   &*(blue!40) & *(blue!40)& *(blue!40)& *(blue!40)\\
*(blue!40)   & *(blue!40)   \\
*(blue!40)  & *(blue!40)   &*(blue!40)   \\
*(blue!40)  & *(blue!40)   &*(blue!40)   &*(blue!40) & *(blue!40)\\
\none\\
\none[\quad  \quad  \quad  \quad  {(5,3,2,7)} ] \\
\end{ytableau}
&&= (-) \,\
\begin{ytableau}
*(blue!40)  & *(blue!40)   &*(blue!40)   \\
*(blue!40)   & *(blue!40)   &*(blue!40) & *(blue!40)& *(blue!40) & *(blue!40)  \\
*(blue!40)  & *(blue!40)   &*(blue!40)   \\
*(blue!40)  & *(blue!40)   &*(blue!40)   &*(blue!40) & *(blue!40)\\
\none\\
\none[\quad  \quad  \quad  {(5,3,6,3)} ] \\
\end{ytableau} 
&&=(+) \,\
\begin{ytableau}
*(blue!40)  & *(blue!40)   &*(blue!40)   \\
*(blue!40)   & *(blue!40)   &*(blue!40) & *(blue!40)\\
*(blue!40)  & *(blue!40)   &*(blue!40) & *(blue!40)  & *(blue!40)\\
*(blue!40)  & *(blue!40)   &*(blue!40)   &*(blue!40) & *(blue!40)\\
\none\\
\none[\quad  \quad    \quad  {(5,5,4,3)} ] \\
\end{ytableau}
\end{align*}
 The
exchange of the last two rows in the Jacobi--Trudi determinant can be achieved by  fixing the last cell on the third row, and the one North-East
to it (here, these two cells are drawn in grey), and letting   the cell to the right of the grey cell in the top row
row fall into the third row.
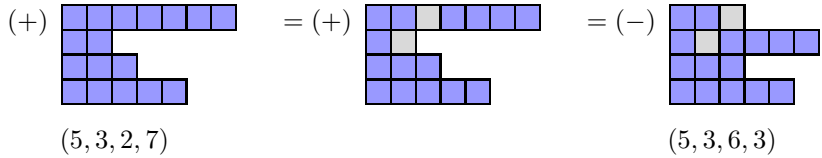
\begin{align*}
\ytableausetup
{boxsize=.9em}\ytableausetup
{aligntableaux=top}
&&(+) \,\
\begin{ytableau}
*(blue!40)  & *(blue!40)   &*(blue!40)   &*(blue!40) & *(blue!40)& *(blue!40)& *(blue!40)\\
*(blue!40)   & *(blue!40)   \\
*(blue!40)  & *(blue!40)   &*(blue!40)   \\
*(blue!40)  & *(blue!40)   &*(blue!40)   &*(blue!40) & *(blue!40)\\
\none\\
\none[\quad   \quad  \quad  {(5,3,2,7)} ] \\
\end{ytableau}
&&=(+) \,\
\begin{ytableau}
*(blue!40)  & *(blue!40)   &*(gray!30)   &*(blue!40) & *(blue!40)& *(blue!40) & *(blue!40) \\
*(blue!40)   & *(gray!30)   \\
*(blue!40)  & *(blue!40)   &*(blue!40)   \\
*(blue!40)  & *(blue!40)   &*(blue!40)   &*(blue!40) & *(blue!40)\\
\end{ytableau}
&&= (-) \,\
\begin{ytableau}
*(blue!40)  & *(blue!40)   &*(gray!30)   \\
*(blue!40)   & *(gray!30)   &*(blue!40) & *(blue!40)& *(blue!40)& *(blue!40)  \\
*(blue!40)  & *(blue!40)   &*(blue!40)   \\
*(blue!40)  & *(blue!40)   &*(blue!40)   &*(blue!40) & *(blue!40)\\
\none\\
\none[\quad   \quad  \quad  {(5,3,6,3)} ] \\
\end{ytableau}
\end{align*}
The result is still not a partition. We repeat the previous procedure with the next two rows:
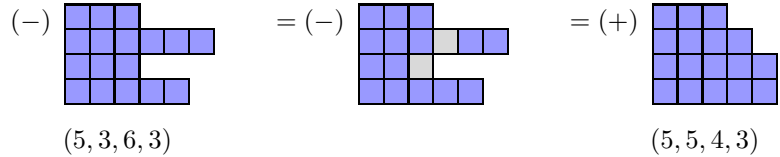
\begin{align*}
&&(-) \,\
\begin{ytableau}
*(blue!40)  & *(blue!40)   &*(blue!40)   \\
*(blue!40)   & *(blue!40)   &*(blue!40) & *(blue!40)& *(blue!40)& *(blue!40)\\
*(blue!40)  & *(blue!40)   &*(blue!40)   \\
*(blue!40)  & *(blue!40)   &*(blue!40)   &*(blue!40) & *(blue!40)\\
\none\\
\none[\quad \quad  \quad  {(5,3,6,3)} ] \\
\end{ytableau}
&&=(-) \,\
\begin{ytableau}
*(blue!40)  & *(blue!40)   &*(blue!40)   \\
*(blue!40)   & *(blue!40)   &*(blue!40) & *(gray!30)& *(blue!40)& *(blue!40)\\
*(blue!40)  & *(blue!40)   &*(gray!30)   \\
*(blue!40)  & *(blue!40)   &*(blue!40)   &*(blue!40) & *(blue!40)\\
\end{ytableau}
&&=(+) \,\
\begin{ytableau}
*(blue!40)  & *(blue!40)   &*(blue!40)   \\
*(blue!40)   & *(blue!40)   &*(blue!40) & *(blue!40)\\
*(blue!40)  & *(blue!40)   &*(blue!40) & *(blue!40)& *(blue!40) \\
*(blue!40)  & *(blue!40)   &*(blue!40)   &*(blue!40) & *(blue!40)\\
\none\\
\none[\quad  \quad \quad  {(5,5,4,3)} ] \\
\end{ytableau}
\end{align*}
%Sometimes this procedure is described as saying that we draw the partition the French way, and then let gravity act.

Note that, if at some point a sequence can not be rectified using this procedure,  then the corresponding Jacobi--Trudi determinant is zero. This situation is illustrated in the next picture, where there is not cell to move from the first to the second row. Remark that this situation occurs because in the corresponding Jacobi--Trudi determinant there is a repeated row.
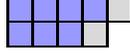
\begin{align*}
\ytableausetup
{boxsize=.9em}\ytableausetup
{aligntableaux=top}
\begin{ytableau}
*(blue!40)  & *(blue!40)  &*(blue!40)  & *(blue!40)   &*(gray!30)   \\
*(blue!40)  & *(blue!40)  & *(blue!40)   & *(gray!30)  \\
\none\\
\none[ \quad \quad \quad s_{(4,5)}=0]
\end{ytableau}
\end{align*}

\subsection{The rectification of $s_{(n,\alpha)}$ }

We need an explicit description of those partitions that can occur as the result of  rectifying an integer sequences, where the first part is the only one allowed to be out-of-order or nonpositive. 

\begin{lemma} \label{characterizationSchur}
Let $\alpha$ be a partition, then after rectification,
 \[
\sum_{n \in \ZZ} s_{(n,\alpha)} [X] = \sum_{\mu} (-1)^{ht(\mu)} s_{\mu}
\]
where we sum over the set of partitions  $\mu$  obtained from $\alpha$ by
\begin{enumerate}

\item Fixing some $k\ge0$,

\item Adding to $\alpha$ a cell to the first $k$ columns, (forming a horizontal strip),

\item Removing a cell from $\alpha$ from  each row $\alpha_i> k$, (forming a vertical strip),

\end{enumerate}  
and where  $ht(\mu)$ be the number of cells in this vertical strip .

\end{lemma}

 Before providing a proof for the result, let us look at an example.

\begin{example}[Rectifications of $\alpha=(n,5,4,3,3)$, with $n \in \ZZ$.]

It is easy to check that if $n<-4$ the Jacobi--Trudi determinant  $s_{(n,5,4,3,3)}=0$, as the first  row of this determinant is identically zero.

We  illustrate all  possible partitions $\mu$ obtained by 
rectifying $\alpha=(k,5,4,3,3),$ for $k \ge -4$ in the following Figure. 

To make explicit the comparison of  $\alpha$ with its rectified partition $\mu$, we draw in white those cells of $\alpha$ that are not in $\mu$ (forming a vertical strip), and in dark blue boxes those cells added that where initially not in $\alpha$ but that are in $\mu$ (forming a horizontal strip). That is to say, the original integer sequence $\alpha$ is the union of the white and the pale blue cells, and its rectification $\lambda$ is the union of both shades of blue.

Finally, we write the integer sequence $\alpha$ below the diagram of  rectified partition it corresponds, and write the parameter $k$ (the index of  Lemma \ref{characterizationSchur}) on top of it. 
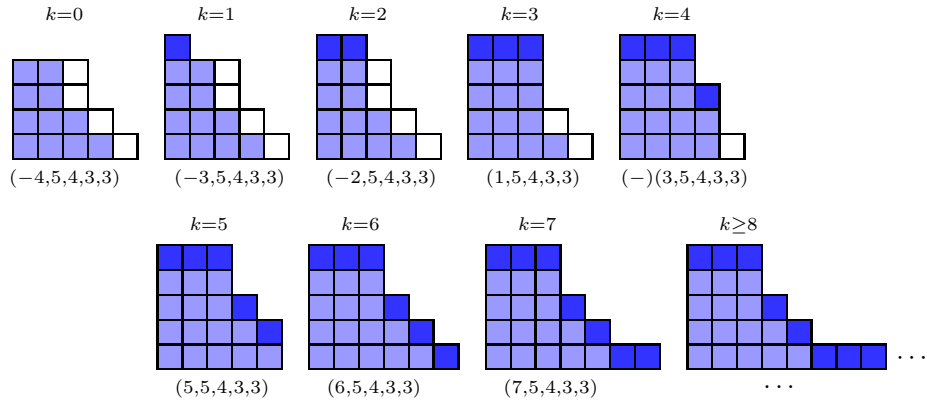
\begin{multline*}
\ytableausetup
{boxsize=.9em}\ytableausetup
{aligntableaux=top}
\begin{ytableau}
\none[\quad  \quad  \quad  ^{k=0}  ] \\
\none\\
*(blue!40) &*(blue!40) & \\
*(blue!40)&*(blue!40) & \\
*(blue!40) &*(blue!40) &*(blue!40)&  \\
*(blue!40) &*(blue!40)  & *(blue!40)   &*(blue!40)   & \\
\none[\quad  \quad  \quad   _{(-4,5,4,3,3)} ] \\
\end{ytableau}
\quad 
\begin{ytableau}
\none[\quad  \quad  \quad  ^{k=1}  ] \\
*(blue!80) \\
*(blue!40) &*(blue!40) & \\
*(blue!40)&*(blue!40) & \\
*(blue!40) &*(blue!40) &*(blue!40)&  \\
*(blue!40) &*(blue!40)  & *(blue!40)   &*(blue!40)   & \\
\none[\quad  \quad  \quad  \quad  _{(-3,5,4,3,3)} ] \\
\none\\
\end{ytableau}
\quad 
\begin{ytableau}
\none[\quad  \quad  \quad  ^{k=2}  ] \\
*(blue!80)  & *(blue!80) \\
*(blue!40) &*(blue!40) & \\
*(blue!40)&*(blue!40) & \\
*(blue!40) &*(blue!40) &*(blue!40)&  \\
*(blue!40) &*(blue!40)  & *(blue!40)   &*(blue!40)   & \\
\none[\quad  \quad  \quad  \quad  _{(-2,5,4,3,3)} ] \\
\none\\
\end{ytableau}
\quad 
\begin{ytableau}
\none[\quad  \quad  \quad  ^{k=3}  ] \\
*(blue!80)  & *(blue!80) & *(blue!80) \\
*(blue!40) &*(blue!40) & *(blue!40)\\
*(blue!40)&*(blue!40) &*(blue!40) \\
*(blue!40) &*(blue!40) &*(blue!40)&  \\
*(blue!40) &*(blue!40)  & *(blue!40)   &*(blue!40)   & \\
\none[\quad  \quad  \quad  \quad  _{(1,5,4,3,3)} ] \\
\none\\
\end{ytableau}
\quad
\begin{ytableau}
\none[\quad  \quad  \quad  ^{k=4}  ] \\
*(blue!80)  & *(blue!80) & *(blue!80) \\
*(blue!40) &*(blue!40) & *(blue!40)\\
*(blue!40)&*(blue!40) &*(blue!40) &*(blue!80) \\
*(blue!40) &*(blue!40) &*(blue!40)&  *(blue!40)\\
*(blue!40) &*(blue!40)  & *(blue!40)   &*(blue!40)   & \\
\none[\quad  \quad  \quad  \quad  _{(-) (3,5,4,3,3)} ] \\
\none\\
\end{ytableau}\\
\quad 
\begin{ytableau}
\none[\quad  \quad  \quad  ^{k=5}  ] \\
*(blue!80)  & *(blue!80) & *(blue!80) \\
*(blue!40) &*(blue!40) & *(blue!40)\\
*(blue!40)&*(blue!40) &*(blue!40) &*(blue!80) \\
*(blue!40) &*(blue!40) &*(blue!40)&  *(blue!40) &*(blue!80)\\
*(blue!40) &*(blue!40)  & *(blue!40)   &*(blue!40)    &*(blue!40) \\
\none[\quad  \quad  \quad  \quad  _{(5,5,4,3,3)} ] \\
\none\\
\end{ytableau}
\quad 
\begin{ytableau}
\none[\quad  \quad  \quad  ^{k=6}  ] \\
*(blue!80)  & *(blue!80) & *(blue!80) \\
*(blue!40) &*(blue!40) & *(blue!40)\\
*(blue!40)&*(blue!40) &*(blue!40) &*(blue!80) \\
*(blue!40) &*(blue!40) &*(blue!40)&  *(blue!40) &*(blue!80)\\
*(blue!40) &*(blue!40)  & *(blue!40)   &*(blue!40)    &*(blue!40) & *(blue!80)\\
\none[\quad  \quad  \quad  \quad  _{(6,5,4,3,3)} ] \\
\none\\
\end{ytableau}
\quad 
\begin{ytableau}
\none[\quad  \quad  \quad  ^{k=7}  ] \\
*(blue!80)  & *(blue!80) & *(blue!80) \\
*(blue!40) &*(blue!40) & *(blue!40)\\
*(blue!40)&*(blue!40) &*(blue!40) &*(blue!80) \\
*(blue!40) &*(blue!40) &*(blue!40)&  *(blue!40) &*(blue!80)\\
*(blue!40) &*(blue!40)  & *(blue!40)   &*(blue!40)    &*(blue!40) & *(blue!80)   & *(blue!80)\\
\none[\quad  \quad  \quad  \quad  _{(7,5,4,3,3)} ] \\
\end{ytableau}
\quad 
\begin{ytableau}
\none[\quad  \quad  \quad  ^{k\ge8}  ] \\
*(blue!80)  & *(blue!80) & *(blue!80) \\
*(blue!40) &*(blue!40) & *(blue!40)\\
*(blue!40)&*(blue!40) &*(blue!40) &*(blue!80) \\
*(blue!40) &*(blue!40) &*(blue!40)&  *(blue!40) &*(blue!80)\\
*(blue!40) &*(blue!40)  & *(blue!40)   &*(blue!40)    &*(blue!40) & *(blue!80)   & *(blue!80)  & *(blue!80) &  \none[\quad \cdots ]\\
\none[\quad  \quad  \quad  \quad \quad  \quad  \quad \ldots  \quad ]\\
\end{ytableau}
\end{multline*}

Note that boxes are deleted from $\alpha$ precisely when $n < \alpha_1$. 
Also, sequence ${(3,5,4,3,3)} $ is the only one that rectifies to a partition negative sign.

\end{example}

\begin{definition}[$\alpha$--removable] Let $\alpha$ be an integer sequence where only the last part is allowed to be out-of-order or nonnegative, and let $\mu$ be its rectification. Draw $\alpha$ and $\mu$ as before. Then, a  cell of  $\alpha \cap \mu $  will be called a {\em $\alpha$--removable} if neither it has a blue cell on top, nor a white cell to its right.

%Let $c $ be a cell  of $\alpha \cap \mu $ (the region draw using pale blue), then the previous lemma shows that either the cell immediately above $c$ belongs to $\mu$, or the cell immediately to the right of $c$ belongs to  $\alpha \setminus \mu$, or both. 
%That is either $c$ has a blue cell on top of it, or a white cell to its right, or both.
%
%A  cell of  $\alpha \cap \mu $ that does not share this property.  That is, the $\alpha$--removable cells are
%those cells of $\alpha \cap \mu $ that neither have a (blue) cell on top of it, nor the cell to its right has been removed (i.e., do not have a white cell to its right).

 \end{definition} 

\begin{remark}
The set of partitions appearing in Lemma   \ref{characterizationSchur} consists of  those partitions that can be obtained from $\alpha$ by adding a horizontal strip, and deleting a vertical strip,  and that leave no $\alpha$--removable corners.
\end{remark}

	\begin{proof} 
	Suppose that $\alpha$ rectifies to a partition $\mu$.
	
	After exchanging rows $r_k$ and $r_{k+1}$, the size of $r_{k+1}$ (now at the $r^{th}$--row) decreases by $1$. This part of the procedure is responsible for the appearance of the vertical strip.  On the other hand, the size of $r_k$ (now at position $r+1$) increases by $1$. 
		
	Originally, there only the first element of the integer sequence was out-of-order. This is the part that keeps moving up  in the rectification process. Eventually, it arrives to a position were the sequence is weakly increasing and positive (otherwise the procedure would have failed).	Looking at the diagram, this has the effect of adding a horizontal strip of $\alpha$. The size of this horizontal strip is the parameter $k$ appearing in the lemma.
	
Finally, notice that  each time that this part goes up, the part that it is  exchanged with losses a cell. As a result, there will be no $\alpha$--removable cells in $\alpha \cap \mu $.
	 	\end{proof}

		\begin{example}   The rectification of $s_ {(-2,5,4,3,3)} $:
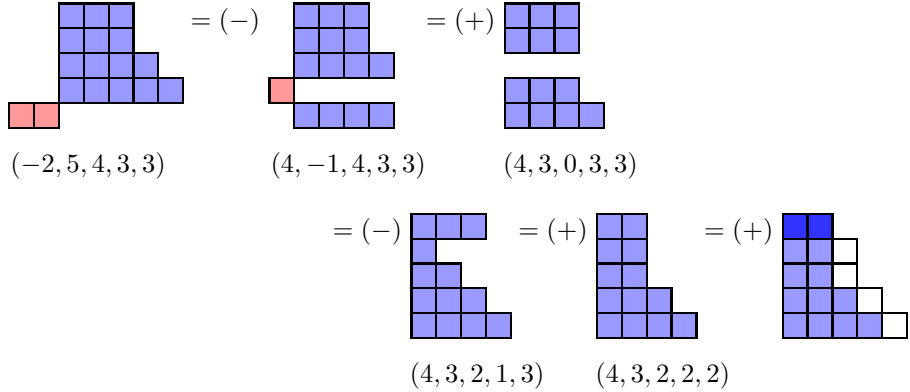
\begin{multline*}
\ytableausetup
{boxsize=.9em}\ytableausetup
{aligntableaux=top}	
	\begin{ytableau}
\none & \none & *(blue!40)  & *(blue!40)&*(blue!40) \\
\none & \none & *(blue!40) &*(blue!40) &*(blue!40) \\
\none & \none & *(blue!40)&*(blue!40) &*(blue!40) &*(blue!40) \\
\none & \none & *(blue!40) &*(blue!40) &*(blue!40) &*(blue!40) &*(blue!40)  \\
*(red!40)& *(red!40) & \none  &\none   & \none    &\none    & \none  \\
\none\\
\none[\quad  \quad  \quad  \quad   \quad  {(-2,5,4,3,3)} ] \\
\none\\
\end{ytableau}
= (-) \,\,
	\begin{ytableau}
 \none & *(blue!40)  & *(blue!40)&*(blue!40) \\
 \none & *(blue!40) &*(blue!40) &*(blue!40) \\
 \none & *(blue!40)&*(blue!40) &*(blue!40) &*(blue!40) \\
 *(red!40) & \none  &\none   & \none    &\none    & \none  \\
 \none & *(blue!40) &*(blue!40) &*(blue!40) &*(blue!40) \\
\none\\
\none[\quad  \quad  \quad  \quad   \quad  {(4,-1,4,3,3)} ] \\
\none\\
\end{ytableau}
= (+) \,\,
	\begin{ytableau}
 *(blue!40)  & *(blue!40)&*(blue!40) \\
 *(blue!40) &*(blue!40) &*(blue!40) \\
 \none  &\none   & \none    &\none    & \none  \\
  *(blue!40) &*(blue!40) &*(blue!40)  \\
 *(blue!40) &*(blue!40) &*(blue!40) &*(blue!40) \\
\none\\
\none[\quad    \quad  \quad   \quad  {(4,3,0,3,3)} ] \\
\none\\
\end{ytableau}\\
= (-) \,\,
	\begin{ytableau}
 *(blue!40)  & *(blue!40)&*(blue!40) \\
 *(blue!40)  \\
*(blue!40) &*(blue!40)  \\
  *(blue!40) &*(blue!40) &*(blue!40)  \\
 *(blue!40) &*(blue!40) &*(blue!40) &*(blue!40) \\
\none\\
\none[\quad  \quad    \quad   \quad  {(4,3,2,1,3)} ] \\
\none\\
\end{ytableau}
= (+) \,\,
	\begin{ytableau}
 *(blue!40)  & *(blue!40)\\
 *(blue!40) &*(blue!40) \\
*(blue!40) &*(blue!40)  \\
  *(blue!40) &*(blue!40) &*(blue!40)  \\
 *(blue!40) &*(blue!40) &*(blue!40) &*(blue!40) \\
\none\\
\none[\quad  \quad  \quad   \quad  {(4,3,2,2,2)} ] \\
\none\\
\end{ytableau}
=(+) \,\,
\begin{ytableau}
*(blue!80)  & *(blue!80) \\
*(blue!40) &*(blue!40) & \\
*(blue!40)&*(blue!40) & \\
*(blue!40) &*(blue!40) &*(blue!40)&  \\
*(blue!40) &*(blue!40)  & *(blue!40)   &*(blue!40)   & \\
\none\\
\end{ytableau}
\end{multline*}
Where we use  the same conventions for colors as in the previous example, and additionally, denoting negative parts with pale red.

\end{example}

It may be interesting to compare those partitions appearing in the previous lemma, with  those of Theorem 1.1 of  \cite{BOR-JA}, where a formula is provided that allows to compute the reduced Kronecker coefficients from the Kronecker coefficients.

%%%%%%%%%%%%%%%%%%%%
%%%%%%%%%%%%%%%%%%%%
%%%%%%%%%%%%%%%%%%%%

\subsection{Operations of alphabets}

We  discuss some basic facts regarding operations on alphabets. First, notice that it is customary to write  a morphism of algebras $A$ from $\Sym$ to some commutative algebra $\mathcal{R}$ as 
 $f \mapsto f[A]$ (rather than $f \mapsto A(f)$),  and consider it as a ``specialization at the virtual alphabet $A$.'' 

Since the power sum symmetric functions $p_k$ ($k \geq 1$) generate $\Sym$ and are algebraically independent, the map 
%\begin{equation}\label{bijection}
$A \mapsto (p_1[A], p_2[A], \ldots)$
%\end{equation}
is a bijection from the set of all morphisms of algebras from $\Sym$ to $\mathcal{R}$ 
to the set of infinite sequences of elements from $\mathcal{R}$. This set of sequences is endowed with its operations of component-wise sum and product, and multiplication by a scalar. 
The bijection %\eqref{bijection} 
is used to lift these operations to the set of morphisms from $\Sym$ to $\mathcal{R}$. 

In this way,  expressions like $f[A+B]$ and $f[AB]$, where $f$ is a symmetric function and $A$ and $B$ are two ``virtual alphabets,'' are defined.  Note that, by definition, for any power sum $p_k$ ($k \geq 1$), virtual alphabets $A$ and $B$, and scalar $z$,
\[
p_k[A+B]=p_k[A]+p_k[B], \qquad p_k[AB]=p_k[A] \cdot p_k[B], \qquad p_k[zA]=z \; p_k[A].
\]
In particular $p_1[A]=A, p_1[A+B]=A+B$, and so on.

The morphism $f \mapsto f^{\perp}=D_f$ associates to $f$ the adjoint of the operator ``multiplication by $f$",  see \cite{LascouxBook, MacdonaldBook}. In particular, this morphism defines a somehow misterious alphabet $ X^{\perp}  = p_1^\perp[X]$ that makes the following identity hold $f^{\perp}[X]=f[X^\perp].$  Later on,  a combinatorial interpretation for both, $p_k^{\perp}[X]$ and $ X^{\perp}$, will be provided.

%%%%%%%%%%%%%%%%%%%%
%%%%%%%%%%%%%%%%%%%%

\subsection{The series $\ss$}
 Let $\ss=\sum_{n=0}^{\infty} h_n$ be the generating function for the complete homogeneous symmetric functions. From the definition of $h_k[A]$ as the sum of all square--free monomials of degree $k$ in alphabet $A$, the following two well--known identities readily follow:
 \begin{align*}
 &\ss[X]=\prod_{x \in X} \frac{1}{1-x}\\
&\ss[X+A]=\ss[X] \ss[A].
\end{align*} 
 
 On the other hand, the Robinson--Schensted--Knuth correspondence provides a combinatorial proof of the identity $\ss[AB]=\sum_{\lambda} s_{\lambda}[A] s_{\lambda}[B] $, known as ``Cauchy's kernel". This is because $\ss[AB]$ is the generating functions for biwords, and because to each such biword the RSK algorithm associates a pair of semistandard tableaux of the same shape (bijectively). The fact that Schur functions are the generating functions for semistardard tableaux completes the argument.

%%%%%%%%%%%%%%%%%%%%
%%%%%%%%%%%%%%%%%%%%

\subsection{Schur functions of a negative alphabet}

We need  a combinatorial interpretation for Schur functions evaluated in a difference of two alphabets. Again, this is carefully done in Lascoux's book,
working with operations on alphabets and Jacobi--Trudi determinants.  Another interesting   approach is given by a formula of Sergeev--Pragacz that expresses a Schur function evaluated at the difference of two alphabet in terms of Vandermonde determinants, a combinatorial proof for this formula can be found in \cite{BergeronGarsia}.
Here we will use its  combinatorial definition in terms of fillings of tableaux, that we present for the sake of completeness.

From the identity $\ss[X+A]=\ss[X] \ss[A]$, setting $A=-X$ we conclude that $1=\sigma[0]=\sigma[X]\sigma[-X]$. Therefore 
\[
\sigma[-X]=\frac{1}{\sigma[X]}=\prod_{x \in X} (1-x)=\prod_{x \in X} (1+(-x))  = \sum (-1)^k e_k[X]
\]

Therefore, applying the the involution  $X \mapsto -X$ to $\sigma$, and we conclude that 
\begin{align}\label{involution}
h_k[X] \mapsto (-1)^k e_k[X]
\end{align}
These two symmetric functions are, respectively, the generating functions for $k$--multisets/sets on alphabet $X$, the most classical instance of combinatorial reciprocity, \cite{StanleyCombRec}. 

% In general, the following well--known theorem describes  Schur function evaluated at differences of alphabets in terms of the combinatorics of tableaux. 

Let $X$ and $Y$ be two alphabets. We are interested in the new alphabet $X-Y$, where we call the elements of $X$ ``positive letters", whereas the ones of $-Y$ are the ``negative letters".
A total order on  $X-Y$  is usually defined by saying that negative letters are smaller than the positive ones, and by breaking ties using  the absolute value of the entries.

\begin{definition}[signed tableau] Let $X-Y$ be an alphabet defined as the difference of two positive alphabets.

A signed tableau $T$ of shape $\lambda$ is a filling of the diagram of $\lambda$ with entries in $X-Y$ such that the entries
are weakly increasing in the rows and columns, and 
\begin{enumerate}
\item Positive letters are strictly increasing on the columns.
\item Negative letters are strictly increasing on the rows.
\end{enumerate}

The weight of a letter depends on its sign. 
Let $a$ be  a positive letter, then $a$ has  weight   $x_a$, and  the corresponding negative letter $-a$ has weight  $-y_a$. The weight 
of a signed tableaux is defined as usual.

We denote by $T_+$ be the subtableau of $T$  positive letters, and by $T_-$ the subtableau of the negative ones.

\end{definition}

\begin{example}
A signed tableau $T$ of shape $(6,5,4,2,1)$:  

\begin{align*}
\ytableausetup
{boxsize=1.4em}
\ytableausetup
{aligntableaux=top}
  \begin{ytableau}
*(blue!20) 5\\
*(blue!20)4&*(blue!20) 4\\
*(red!20) {-1} &*(red!20) -3 & *(blue!20)3 & *(blue!20)3\\
*(red!20) -1&*(red!20) -3& *(blue!20) 1 & *(blue!20)2 &*(blue!20) 3\\
*(red!20) -1 &*(red!20) -2&*(red!20) -3& *(blue!20)1   & *(blue!20)2  & *(blue!20)2   \\
\end{ytableau}
\end{align*}
where the subtableau $T_+$ is shaded blue, whereas $T_-$ is shaded red.
The signed tableaux $T$ has weight:
\begin{align*}
x_1^2 x_2^3 x_3^3 x_4^2 x_5 (-y_1)^3(-y_2)(-y_3)^3=-x_1^2 x_2^3 x_3^3 x_4^2 x_5\, y_1^3 y_2 y_3^3
\end{align*}

\end{example}

\begin{theorem}[A Schur function evaluated in the difference of two alphabets]\label{Schurdiff}
Let $\lambda$ be a partition, and let $X$ and $Y$ be two alphabets,  then
\[
s_\lambda[X-Y] = \sum_{\substack{T=T_-\cup T_+}} x^{T_+} y^{T_-}
\]
where  we are summing over all possible signed tableaux $T=T_-\cup T_+$ of shape $\lambda$.

\end{theorem}

\begin{proof}
Theorem \ref{Schurdiff} follows from the identity $s_\lambda[X+Y]=\sum_{\mu \in \lambda} s_\mu[X] s_{\lambda /\mu}[Y]$  (immediate from the combinatorial definition of Schur function in terms of fillings of tableaux), combined with Eq. \ref{involution}, and the dual Jacobi-Trudi identity (the determinant expression for a Schur function on the elementary symmetric basis). 
\end{proof}

If the negative alphabet $-Y=\{-t\}$  has just one letter, then condition (2) implies that  $T_-$ defines a vertical strip. Moreover, since  $s_\lambda[X-Y]$ is a symmetric function, we can use any total order for the letters of $X-Y$  in its combinatorial definition. Setting negative letters to be bigger than the positive ones makes it transparent that  $s_\lambda[X-Y]$ is the sum of all Schur functions $t^\ell S_\alpha[X]$, where $\alpha$ can be obtained from $\lambda$ by removing a vertical strip of size $\ell$. We have obtained the following corollary:

\begin{corollary} \label{NegativeSchur}
Let $\alpha$ be a partition, then 
\[
 s_{\alpha} \big[X -1/t \big] = \sum_{k\ge0 } \, \, \sum_{\substack{\lambda \text{ partition } \\\alpha\setminus \lambda \text{ $k$--vertical strip}}} \big(-1/t \big)^k s_{\lambda}[X]
\]
\end{corollary} 

From the combinatorial definition of a Schur function evaluated in a negative alphabet Cauchy's dual identity follows:
\[
\ss[-AB]=\sum_{\lambda} (-1)^{|\lambda|} s_{\lambda'}[A] s_{\lambda}[B] 
\]

%%%%%%%%%%%%%%%%%%%%%%%%%%%%%%%%%%%%%%%%
%%%%%%%%%%%%%%%%%%%%%%%%%%%%%%%%%%%%%%%%
%%%%%%%%%%%%%%%%%%%%%%%%%%%%%%%%%%%%%%%%

\subsection{The alphabet $X^\perp$}

The object of this section is to provide a combinatorial interpretation for  the specialization $p_n^\perp[X]   = p_n[X^\perp]$ .
We recall the following classical result,   \cite[p. 25]{LascouxBook}.
\begin{lemma} \label{MNruleDer}

Let $\alpha=(\alpha_1, \alpha_2, \cdots, \alpha_n)$ be a sequence of positive integers
\[
p^\perp_i s_{\alpha} = \sum_{k=1}^n s_{(\alpha_1, \cdots, \alpha_k-i, \cdots, \alpha_n)}
\]
%The operator $p^\perp_i $ acts as a deviation.
\end{lemma}
\begin{proof}
It is immediate that $p^\perp_i h_k=h_{k-i}$, as the coefficients of $h_n$ in the power sum basis are one, for all $p_{\lambda}$, with $\lambda \vdash n$.

On the other hand, if we develop the Jacobi--Trudi determinant, and then apply the
product rule, the result follows easily by regrouping together all the summand where the same $h_k$ has been derived.
\end{proof}

\begin{example}[The effect of the perp operator] \label{MNex}

\[
p^\perp_3 s_{(6,5,4,2,1)} = s_{(3,5,4,2,1)}+s_{(6,2,4,2,1)}+s_{(6,5,1,2,1)}+s_{(6,5,4,-1,1)}+s_{(6,5,4,2,-2)}
\]
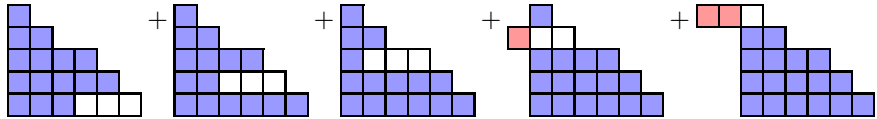
\begin{align*}
\ytableausetup
{boxsize=.8em}\ytableausetup
{aligntableaux=top}
  \begin{ytableau}
*(blue!40)\\
*(blue!40)& *(blue!40)\\
*(blue!40) &*(blue!40) & *(blue!40) & *(blue!40)\\
*(blue!40)&*(blue!40) &*(blue!40) & *(blue!40) & *(blue!40)\\
*(blue!40) &*(blue!40) &*(blue!40)&   &   &    \\
\end{ytableau}
+
\begin{ytableau}
*(blue!40)\\
*(blue!40)& *(blue!40)\\
*(blue!40) &*(blue!40) & *(blue!40) & *(blue!40)\\
*(blue!40)&*(blue!40) & & &\\
*(blue!40) &*(blue!40) &*(blue!40)&*(blue!40)  & *(blue!40) & *(blue!40)  \\
\end{ytableau}+
\begin{ytableau}
*(blue!40)\\
*(blue!40)& *(blue!40)\\
*(blue!40) && &\\
*(blue!40)&*(blue!40) &*(blue!40) & *(blue!40) & *(blue!40)\\
*(blue!40) &*(blue!40) &*(blue!40)&*(blue!40)  & *(blue!40) & *(blue!40)  \\
\end{ytableau}+
\begin{ytableau}
\none& *(blue!40)\\
*(red!40) & &\\
\none& *(blue!40) &*(blue!40) & *(blue!40) & *(blue!40)\\
\none& *(blue!40)&*(blue!40) &*(blue!40) & *(blue!40) & *(blue!40)\\
\none& *(blue!40) &*(blue!40) &*(blue!40)&*(blue!40)  & *(blue!40) & *(blue!40)  \\
\end{ytableau}
+
\begin{ytableau}
*(red!40)&*(red!40)&\\
\none&\none&*(blue!40)& *(blue!40)\\
\none&\none&*(blue!40) &*(blue!40) & *(blue!40) & *(blue!40)\\
\none&\none&*(blue!40)&*(blue!40) &*(blue!40) & *(blue!40) & *(blue!40)\\
\none&\none&*(blue!40) &*(blue!40) &*(blue!40)&*(blue!40)  & *(blue!40) & *(blue!40)  \\
\end{ytableau}
\end{align*}

If we rectify the resulting sequences we then obtain:
\[
p^\perp_3 s_{(6,5,4,2,1)} = -s_{(4,4,4,2,1)}-s_{(6,3,3,2,1)}+s_{(6,5,4)}-s_{(6,5,4,2,-2)}
\]
The the partitions indexing nonzero Schur functions are precisely
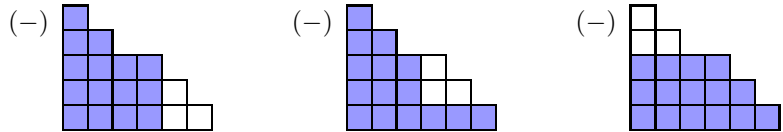
\begin{align*}
\ytableausetup
{boxsize=.9em}\ytableausetup
{aligntableaux=top}
(-)\,\,\,
  \begin{ytableau}
*(blue!40)\\
*(blue!40)& *(blue!40)\\
*(blue!40) &*(blue!40) & *(blue!40) & *(blue!40)\\
*(blue!40)&*(blue!40) &*(blue!40) & *(blue!40) & \\
*(blue!40) &*(blue!40) &*(blue!40)&  *(blue!40) &   &    \\
\end{ytableau}
\quad\quad\quad
(-)\,\,\,
\begin{ytableau}
*(blue!40)\\
*(blue!40)& *(blue!40)\\
*(blue!40) &*(blue!40) & *(blue!40) & \\
*(blue!40)&*(blue!40) &*(blue!40) & &\\
*(blue!40) &*(blue!40) &*(blue!40)&*(blue!40)  & *(blue!40) & *(blue!40)  \\
\end{ytableau}
%+
%\begin{ytableau}
%*(blue!40)\\
%*(blue!40)& *(blue!40)\\
%*(blue!40) && &\\
%*(blue!40)&*(blue!40) &*(blue!40) & *(blue!40) & *(blue!40)\\
%*(blue!40) &*(blue!40) &*(blue!40)&*(blue!40)  & *(blue!40) & *(blue!40)  \\
%\end{ytableau}
\quad\quad\quad
(-)\,\,\,
\begin{ytableau}
\\
& \\
*(blue!40) &*(blue!40) & *(blue!40) & *(blue!40)\\
*(blue!40)&*(blue!40) &*(blue!40) & *(blue!40)&*(blue!40)\\
*(blue!40) &*(blue!40) &*(blue!40)&*(blue!40)  & *(blue!40) & *(blue!40)  \\
\end{ytableau}
\end{align*}
because  sumands 3 and 5 rectify to zero.
\end{example}

As the reader probably suspects, if we rectify the partitions appearing in Lemma  \ref{MNruleDer}
we obtain the classical Murnaghan--Nakayama rule. Details can be found in  \cite{LascouxBook}. On the other hand, the Murnaghan--Nakayama rule can be  derived just as easily  from the previous lemma using the combinatorial 
interpretation just provided.

The definition of specialization of a symmetric function at a virtual alphabet $A$ needs the understanding of $p_n[A]$, for all $n\ge0$.
From a combinatorial point of view,  the effect of the $p_i^\perp$ operator on a Jacobi--Trudi determinant
is to subtract $n$ cells from each of the rows of ${\alpha}$, one at a time. 
On the other hand, if we fix an alphabet $X$, then the identity $p_n^\perp[X]   = p_n[X^\perp]$  provides us with a  plausible combinatorial 
	interpretation for  $(x_k^n)^\perp$:  we say that  $(x_k^n)^\perp$
	acts on $ s_{\alpha}[X]$ by deleting $n$ cells from the $k^{th}$ row of $\alpha$. This makes sense as $ s_{\alpha}[X]$ 
	is nonzero only if $|X| \ge \ell(\lambda)$ so we will be  subtracting $k$ cells for each of the rows, successively.

%%%%%%%%%%%%%%%%%%%%
%%%%%%%%%%%%%%%%%%%%
%%%%%%%%%%%%%%%%%%%%
\subsection{The vertex operator $\Gamma_{(t|X)}$}

 Recall that the vertex operator $\Gamma_{(t|X)}$ is defined by
 $
 s_\alpha =\sum_{n \in \ZZ} s_{(n,\alpha)} [X] t^n.
$ We present a sign--reversing involution that shows that the following identity holds.
% We will show combinatorially that $ \sigma[tX]$ divides $\Gamma_{(t|X)} s_\alpha $. 

\begin{lemma}
Let $\alpha$ be a partition.
\begin{align}
\Gamma_{(t|X)} s_\alpha = \sigma[tX] s_{\alpha}\big[x-1/t\big]
\end{align}
 \end{lemma}

\begin{proof}

%Recall that $\sigma[tX]$ is the generating functions for the $h_n$'s. Therefore, the product of  $\sigma[tX] s_{\alpha}\big[x-1/t\big]$ is described by Pieri's rule.

First, we apply  Lemma \ref{NegativeSchur} followed with Pieri's rule to the expression  $\sigma[tX] s_{\alpha}\big[x-1/t\big]$  (which is valid since the sequences involved are always partitions).
The partitions  indexing a nonzero Schur function that appear as summands are precisely those that can be obtained by first deleting a vertical strip, and then adding a horizontal one. The sign of the corresponding
Schur function is given by  $(-1)^{|\text{vertical strip}|}$. 

Since this is a signed sum, there are plenty of potential cancelations in this sum. Let us look at them closely. Let $\mu$ be a partition that indexes on of the summands of  $\sigma[tX] s_{\alpha}\big[x-1/t\big]$ (before carrying out any cancellation).  We will refer to such a summand directly by the indexing  partition $\mu$.

We classify the resulting partitions into two classes. Those partitions $\mu$ that have a   $\alpha$--removable corner, and those that have not.

Suppose that there are $\alpha$--removable corner in $\mu$, and denote $c$ be the leftmost one.
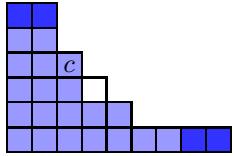
\begin{align*}
\ytableausetup
{boxsize=.9em}\ytableausetup
{aligntableaux=top}
  \begin{ytableau}
 *(blue!80)  & *(blue!80)\\
*(blue!40) & *(blue!40)\\
*(blue!40)& *(blue!40) &   *(blue!40)c\\
*(blue!40) &*(blue!40) & *(blue!40) & \\
*(blue!40)&*(blue!40) &*(blue!40) & *(blue!40) & *(blue!40) \\
*(blue!40) &*(blue!40) &*(blue!40)&  *(blue!40) & *(blue!40)     &  *(blue!40) & *(blue!40) &*(blue!80)  & *(blue!80) \\
\end{ytableau}
\end{align*}
The summand indexed by $\mu$ appears indexing twice in  $\sigma[tX] s_{\alpha}\big[x-1/t\big]$: Once where  corner $c$ of $\mu$  was  not removed as part of  the vertical strip. The other one, when $c$ was indeed removed as part of the vertical strip, but then added again by the horizontal strip. Our involution will pair this two tableaux. 

On the other hand, if i there is no $\alpha$--removable corner in $\mu$, then it will be paired to itself.

Notice that if two different tableaux are paired together, then they   have the same weight (as the factors $t$ and $1/t$ will cancel each other), but  their signs will differ. Therefore the two corresponding summands will cancel each other out. 
 On the other hand,  those partitions with no $\alpha$--removable do not cancel out.

We conclude this argument using Lemma \ref{characterizationSchur} we it was showed that, after rectification, those partitions (the ones without $\alpha$--removable corners) are precisely the Schur functions appearing in $\Gamma_{(t|X)} s_\alpha$.
%Finally, note that the sign of the resulting Schur functions coincide on both definitions.
\end{proof}

We finish this note with a useful form of this identity, see \cite{ScharfThibonWybourne, ThibonHopf, BMOR}. 
Let $U_{\alpha}$ be the operator multiplication by $s_{\alpha}$, and let $D_{\beta}$ be the skewing operator. That is, these adjoint operators are defined on the Schur basis by  $U_{\alpha}(s_\lambda)= s_\lambda s_{\alpha} $, and $D_{\beta}(s_\lambda)=s_{\lambda / \beta}$.  From Pieri's rule, and its dual, it follows that the vertex operator, evaluated at $t=1$, can be rewritten in the following elegant and well--known base--free way:
\begin{equation}\label{Gamma-VO}
\Gamma_1=\left( \sum_{i=0}^{\infty} U_{(i)}\right) \left( \sum_{j=0}^{\infty} (-1)^j \D{(1^j)}\right)
\end{equation}

%%%%%%%%%%%%%%%%%%%%
%%%%%%%%%%%%%%%%%%%%
%%%%%%%%%%%%%%%%%%%%
%\section{FALTA}
%
%
%\begin{lemma} \label{properties:ss}
%Let $A$ and $B$  be any two alphabets, and $f$ and $g$ be any two symmetric functions. Then we have the following identities.
%\begin{align}
%&\D{\ss[AX]}(f[X])=f[X+A] && \label{sum}\\
%%&\scalar{\ss[AB]}{g[B]}_B=g[A] && \text{ (Reproducing Kernel)}\label{reproducing}
%\end{align}
%\end{lemma}
%
%

\bibliographystyle{plain}
\bibliography{gamma}

\end{document}